\documentclass[a4paper,12pt]{amsart}

\usepackage[headings]{fullpage}

\usepackage{amsfonts,graphics,amsmath,mathrsfs,amsthm,amscd,amssymb,latexsym,euscript,enumerate}
\usepackage{epsfig}
\usepackage{flafter}
\usepackage[all,cmtip,line]{xy}
\usepackage{array}
\usepackage[english]{babel}
\usepackage{overpic}
\usepackage{subfig}
\usepackage{multirow}
\usepackage{microtype}
\usepackage{wrapfig}
\usepackage{longtable}
\usepackage{supertabular}
\usepackage[shortlabels]{enumitem}
\usepackage{tikz}
\usetikzlibrary{positioning}
\usepackage{tikz-cd}
\usepackage{float}
\allowdisplaybreaks

\usepackage[dvipsnames,svgnames,table]{xcolor}
\usepackage{graphicx}
\usepackage{hyperref}
\hypersetup{
    colorlinks=true,
    linkcolor=blue,
    citecolor=blue,
    filecolor=blue,
    urlcolor=blue
}
\usepackage{cleveref}
\usepackage[textsize=footnotesize]{todonotes}

\newtheorem{theorem}{Theorem}[section]
\newtheorem{lemma}[theorem]{Lemma}
\newtheorem{proposition}[theorem]{Proposition}

\newtheorem{construction}[theorem]{Construction}
\newtheorem{question}[theorem]{Question}
\newtheorem*{theorem*}{Theorem}

\theoremstyle{plain}

\theoremstyle{definition}
\newtheorem{definition}[theorem]{Definition}
\newtheorem{definition-lemma}[theorem]{Definition-Lemma}

\newtheorem{remark}[theorem]{Remark}

\numberwithin{equation}{section}

\newcommand{\C}{\mathbb{C}}
\newcommand{\R}{\mathbb{R}}
\newcommand{\Z}{\mathbb{Z}}

\newcommand{\Q}{\mathbb{Q}}
\newcommand{\OO}{\mathcal{O}}

\newcommand{\G}{\mathbb{G}}
\def\P{\mathbb{P}}

\DeclareMathOperator{\Nef}{Nef}

\DeclareMathOperator{\Eff}{Eff}

\DeclareMathOperator{\codim}{codim}

\def\Pic{\operatorname{Pic}}

\def\Spec{\operatorname{Spec}}

\def\Mov{\operatorname{Mov}}

\def\Supp{\operatorname{Supp}}

\def\NE{\operatorname{NE}}

\usepackage{mathtools}

\DeclarePairedDelimiterX{\norm}[1]{\lVert}{\rVert}{#1}

\title[Mori dream fibers and the geometric generic fiber]{Mori dream fibers and the geometric generic fiber}

\begin{document}

\author[D.-W. Lee]{Dae-Won Lee}
\author[M. Nagaoka]{Masaru Nagaoka}
\address[Dae-Won Lee]{School of Mathematics, Korea Institute for Advanced Study, 85 Hoegiro, Dongdaemun-gu, Seoul 02455, Republic of Korea}
\email{daewonlee@kias.re.kr}
\address[Masaru Nagaoka]{Gakushuin University, 1-5-1 Mejiro, Toshima-ku, Tokyo 171-8588, Japan}
\email{masaru.nagaoka@gakushuin.ac.jp}

\date{\today}

\begin{abstract}
We construct a smooth projective family of rational surfaces over $\G_{\mathrm m,\Z}$. The Mori dream property of a fiber is determined by the torsion of the normal bundle of an anticanonical cycle. Over $\mathbb C$, the locus of Mori dream fibers is Zariski dense. For every prime $p$, every geometric fiber over a closed point of the reduction modulo $p$ is a Mori dream surface, whereas the geometric generic fiber is not a Mori dream space. In either setting, no restriction to a nonempty open subset is a Mori dream morphism. We also prove that, over any algebraically closed field, a projective fibration becomes a Mori dream morphism after shrinking the base whenever the set of points with Mori dream fibers is not contained in a countable union of proper closed subsets.
\end{abstract}

\maketitle

%\tableofcontents

%%%%%%%%%%%%%%%%%%%%%%%%%%%%%%%%%%%%%%%%%%%%%%%%%%%%%
\section{Introduction}

Hu--Keel introduced Mori dream spaces to study projective varieties whose birational contractions are described by a finite decomposition of the effective cone into rational polyhedral chambers. On a Mori dream space defined over $\C$, the minimal model program can be run for every $\Q$-Cartier divisor, and every sequence of flips terminates, see \cite[Definition 1.10 and Proposition 1.11]{HK}. For a normal $\Q$-factorial projective variety $Y$ over $\C$ satisfying $\Pic(Y)_{\Q}=N^1(Y)_{\Q}$, the Mori dream property is equivalent to the finite generation of a Cox ring \cite[Proposition 2.9]{HK}.

An important class of examples consists of $\Q$-factorial projective varieties of Fano type over $\C$. Such a variety admits an effective $\Q$-divisor $\Delta$ for which $(Y,\Delta)$ is klt and $-(K_Y+\Delta)$ is ample. It is well known that varieties of Fano type are Mori dream spaces \cite[Corollary 1.3.2]{BCHM}.

Let $f\colon X\to T$ be a projective fibration between normal varieties. Throughout, by a projective fibration, we mean a projective surjective morphism satisfying $f_*\OO_X=\OO_T$. For a Mori dream morphism, the relative movable cone is covered by the nef cones of finitely many small $\Q$-factorial modifications, and each of these nef cones is generated by finitely many relatively semiample divisor classes, see \cite[Definition 3.1]{Ohta} and Definition \ref{def:mdm}. Over $\C$, Mori dream morphisms admit a relative minimal model program for every $\Q$-Cartier divisor, see \cite[Theorem 5.6]{Ohta}. Under the $\Q$-factoriality and relative Picard conditions, they are also characterized by finite generation of a relative Cox sheaf \cite[Theorem 1.1]{Ohta}.

In \cite{CLZ}, Choi--Li--Zhou ask whether the Fano type property of the fibers over a Zariski dense subset implies the relative Fano type property after shrinking the base \cite[Question 1.1]{CLZ}. They establish this implication under additional assumptions, including the weak Fano condition, a very general choice of the points, or bounded klt complements \cite[Theorem 1.2]{CLZ}. Recent work of Kim gives further affirmative results for families with fibers of Fano type. If the anticanonical volumes of the Fano type fibers are equal to a fixed positive number on a Zariski dense subset of the base, then the geometric generic fiber is of Fano type \cite[Theorem 1.2]{Kim25}. For families of surfaces, the same conclusion holds for a Zariski dense subset of Fano type fibers without any assumption on the anticanonical volumes \cite[Theorem 1.3]{Kim25}.

Choi--Li--Zhou also ask the analogous question for Mori dream spaces. 

\begin{question}[\protect{\cite[Question 3.2]{CLZ}}]\label{que}
    Let $f\colon X\to T$ be a projective fibration between normal varieties. Suppose that there exists a Zariski dense subset $S\subset T$ such that $X_s$ is a Mori dream space for every $s\in S$. Does there exist a nonempty open subset $U\subset T$ such that the restricted morphism $X_U\to U$ is a Mori dream morphism?
\end{question}

In this paper, we give a negative answer to Question \ref{que}. In Construction \ref{con:family}, we construct a smooth projective family of rational surfaces $\Pi\colon\mathfrak X\to\mathbb G_{\mathrm m,\Z}$. Proposition \ref{prop:fiberwise} describes the negative curves of every geometric fiber and shows that the Mori and nef cones are rational polyhedral and independent of the parameter with respect to the fixed bases of total transforms. Proposition \ref{prop:fiberwise} also shows that a fiber is a Mori dream space precisely when its parameter has finite multiplicative order.

After base change to $\C$, let $X_q$ denote the fiber over $q\in \C^{\times}$. Proposition \ref{prop:complex-counterexample} gives
\[
\{q\in\mathbb C^\times\mid X_q\text{ is a Mori dream space}\}=\mu_\infty,
\]
where $\mu_\infty$ denotes the group of roots of unity in $\mathbb C^\times$. Thus, the locus of Mori dream fibers is Zariski dense and nonconstructible. Proposition \ref{prop:finite-field-counterexample} gives a second counterexample after reduction modulo every prime $p$: every geometric fiber over a closed point of $\mathbb G_{\mathrm m,\mathbb F_p}$ is a Mori dream surface, whereas the geometric generic fiber is not a Mori dream space. In both settings, no restriction to a nonempty open subset of the base is a Mori dream morphism.

The following theorem gives a positive answer under a stronger genericity assumption on the set of points with Mori dream fibers.

\begin{theorem}\label{thm:main}
Let $k$ be an algebraically closed field of arbitrary characteristic, and let $f\colon X\to T$ be a projective fibration between normal integral $k$-varieties.  Let $S\subset T(k)$ be a subset such that $X_s$ is a Mori dream space for every $s\in S$. Assume that $S$ is not contained in a countable union of proper Zariski closed subsets of $T$. Then there is a nonempty open subset $U\subset T$ such that $X_U\to U$ is a Mori dream morphism.

In particular, if $T$ is a curve and $X_s$ is a Mori dream space for uncountably many points $s\in T(k)$, then there is a nonempty open subset $U\subset T$ such that $X_U\to U$ is a Mori dream morphism.
\end{theorem}

Theorem \ref{thm:main} is complementary to \cite[Theorem 1.2]{CLZ}. We assume only that the relevant fibers are Mori dream spaces, while imposing the stronger genericity condition that $S$ is not contained in a countable union of proper closed subsets.

We briefly explain the proof of Theorem \ref{thm:main}. By Lemma \ref{lem:countable-model}, we choose a countable algebraically closed subfield $k_0\subset k$, a model $f_0\colon X_0\to T_0$ over $k_0$, and a point $s\in S$ whose image in $T_0$ is the generic point. Lemma \ref{lem:field-comparison} gives a compatible isomorphism between the two algebraically closed extensions of $k_0(T_0)$ arising from $s$ and from the geometric generic point. Consequently, a base change of $X_s$ is isomorphic to the geometric generic fiber $X_{\overline\eta}$, and hence $X_{\overline\eta}$ is a Mori dream space.

Proposition \ref{prop:generic-criterion} then spreads the Mori dream structure of the geometric generic fiber over a nonempty open subset of $T$. More precisely, the finitely many small $\Q$-factorial modifications, the semiample generators of their nef cones, and generating sections are extended over an open subset by Lemma \ref{lem:models-open}; after further shrinking, the birational maps remain small by Lemma \ref{lem:small-open}, while Lemma \ref{lem:qfactor-open} identifies the relative divisor classes with those on the generic fiber. These data determine the relative nef cones and the decomposition of the movable cone required in Definition \ref{def:mdm}.

The rest of the paper is organized as follows. In Section \ref{sec:prelim}, we collect the definitions, descent and extension results, and results on surfaces used in the proofs. In Section \ref{sec:main}, we construct the family, establish the two counterexamples, and prove Theorem \ref{thm:main}.

\section*{Acknowledgement}
The authors would like to express their gratitude to Prof. Sung Rak Choi for his valuable comments. The first author is partially supported by Samsung Science and Technology Foundation under Project Number SSTF-BA2302-03, by the National Research Foundation of Korea (No. RS-2025-00513064) and by a KIAS Individual Grant (MG111201) at Korea Institute for Advanced Study. The second author is supported by JSPS KAKENHI Grant Number JP21K13768.

\section{Preliminaries}\label{sec:prelim}

\subsection{Mori dream spaces and Mori dream morphisms}
We begin by fixing the notions used throughout the paper. 

Let $F$ be a field, and let $g\colon Y\to V$ be a projective fibration between normal quasi-projective varieties over $F$. 

A prime divisor $E$ on $Y$ is called \emph{horizontal over} $V$ when $\overline{g(E)}=V$, and \emph{vertical over} $V$ when $\overline{g(E)}\neq V$. An $\R$-divisor is horizontal (resp. vertical) over $V$ if every prime component of its support is horizontal (resp. vertical) over $V$.

We define the relative Picard group by $\Pic(Y/V)\coloneqq \Pic(Y)/g^{\ast}\Pic(V)$, where $g^{\ast}\Pic(V)$ denotes the image of the pullback homomorphism. We write $L_1\equiv_{V} L_2$ if two line bundles $L_1,L_2$ on $Y$ are numerically equivalent over $V$, i.e., $\deg(L_1|_C)=\deg(L_2|_C)$ for every integral projective curve $C\subset Y$ contracted by $g$. We let $N^1(Y/V)\coloneqq \Pic(Y)/\equiv_{V}$. For an abelian group $M$, we write $M_{\Q}\coloneqq M\otimes_{\Z} \Q$ and $M_{\R}\coloneqq M\otimes_{\Z} \R$. The real vector space $N^1(Y/V)_{\R}$ is the relative numerical divisor space, and all relative divisor cones are considered in this space. Pullbacks of line bundles on $V$ are numerically trivial over $V$, and hence there is a natural surjection $\Pic(Y/V)\to N^1(Y/V)$. We omit $V$ when $V=\Spec F$.

A $\Q$-Cartier divisor $D$ on $Y$ is called $g$-semiample if some positive multiple of $D$ is Cartier and relatively globally generated. A birational map $\psi\colon Y\dashrightarrow Y'$ over $V$ is called a \emph{small $\Q$-factorial modification} if $Y$ and $Y'$ are both $\Q$-factorial and $\psi$ is an isomorphism in codimension one. We write $\Nef(Y/V)$ for the nef cone over $V$ and $\overline{\Eff}(Y/V)$ for the pseudoeffective cone over $V$. 

A Cartier divisor $D$ on $Y$ is called \emph{$g$-movable} if
\[
\codim_Y\Supp\mathrm{Coker}\bigl(g^*g_*\OO_Y(D)\longrightarrow\OO_Y(D)\bigr)
 \geq 2.
\]
We denote by $\Mov(Y/V)$ the cone generated by the numerical classes of $g$-movable Cartier divisors. The notions of relative movability in \cite{CLZ} and \cite{Ohta} differ slightly. In \cite{CLZ}, the codimension condition is imposed on the Cartier divisor $D$ itself, whereas in \cite{Ohta}, it is imposed on a positive multiple of a $\Q$-divisor. The two definitions nevertheless define the same relative movable cone, since if $mD$ is movable for a positive integer $m$, then
\[
 [D]=\frac{1}{m}[mD].
\]

\begin{definition}\label{def:mdm}
The morphism $g\colon Y\to V$ is a \emph{Mori dream morphism} if the following
conditions hold.
\begin{enumerate}[(1)]
 \item The variety $Y$ is $\Q$-factorial.
 \item One has $\Pic(Y/V)_{\Q}=N^1(Y/V)_{\Q}$.
 \item The cone $\Nef(Y/V)$ is rational polyhedral and generated by finitely many
 $g$-semiample divisor classes.
 \item There are finitely many small $\Q$-factorial modifications
 $\phi_i\colon Y\dashrightarrow Y_i$ over $V$ such that every
 $Y_i\to V$ satisfies (1)--(3) and
 \[
  \Mov(Y/V)=\bigcup_i\phi_i^*\Nef(Y_i/V).
 \]
\end{enumerate}
When $V=\Spec F$, we call $Y$ a \emph{Mori dream space} in the sense of Definition \ref{def:mdm}.
\end{definition}

\begin{remark}\label{rem:nef-semiample}
If $g\colon Y\to V$ satisfies conditions (2) and
(3) of Definition \ref{def:mdm}, then every $g$-nef
$\Q$-Cartier divisor on $Y$ is $g$-semiample.

Indeed, condition (2) lifts a rational expression for the numerical class of a nef divisor in terms of the semiample generators from $N^1(Y/V)_{\Q}$ to $\Pic(Y/V)_{\Q}$. After clearing denominators and replacing the generators by positive multiples which are relatively globally generated, a positive multiple of the original nef divisor is relatively globally generated, see \cite[Proposition 3.2]{Ohta}.
\end{remark}

Ohta formulates Definition \ref{def:mdm} over $\mathbb C$, and we use the same four conditions in Definition \ref{def:mdm} over an arbitrary field. Over
$\mathbb C$, Definition \ref{def:mdm} agrees with \cite[Definition 3.1]{Ohta}. If $F$ is algebraically closed and $V=\Spec F$, then Definition \ref{def:mdm} coincides with the Hu--Keel definition \cite[Definition 1.10]{HK}. On the other hand, even if $F$ is algebraically closed, the absolute case $V=\Spec F$ of Definition \ref{def:mdm} need not coincide with Okawa's definition. Lemma \ref{lem:picard-conditions} gives the comparison required in the proof of Proposition \ref{prop:generic-criterion}.

\subsection{Descent and extension over an open subset}\label{subsec:descent}
The following results allow us to pass between the geometric generic fiber, the generic fiber, and models over a nonempty open subset of the base. For normal projective varieties over arbitrary fields, we use Okawa's definition of a Mori dream space \cite[Definition 2.1]{Okawa}.

\begin{proposition}[{\protect\cite[Proposition 2.11]{Okawa}}]\label{prop:okawa}
Let $K\subset L$ be a field extension and let $Y$ be a normal projective
variety over $K$. If $Y_L$ is a Mori dream space in the sense of Okawa, then $Y$ is a Mori dream space in the sense of Okawa.
\end{proposition}

The next lemma compares the Picard condition in Definition \ref{def:mdm} with the condition on $\Pic^0$ in \cite{Okawa}.

\begin{lemma}\label{lem:picard-conditions}
    Let $F$ be a field, and let $Y$ be a normal projective variety over $F$ with $H^0(Y,\OO_Y)=F$.
    \begin{enumerate}[(1)]
        \item If $\dim \Pic^0_{Y/F}=0$, then the natural map
        \[
        \Pic(Y)_{\Q}\longrightarrow N^1(Y)_{\Q}
        \]
        is an isomorphism.
        \item Assume in addition that $F$ is uncountable and algebraically closed. If the map $\Pic(Y)_{\Q}\to N^1(Y)_{\Q}$ is an isomorphism, then $\dim \Pic^0_{Y/F}=0$.
    \end{enumerate}
    Consequently, over an uncountable algebraically closed field, the absolute case of Definition \ref{def:mdm} coincides with Okawa's definition of a Mori dream space.
\end{lemma}
\begin{proof}
Let $\Pic^0_{Y/F}$ be the identity component of the Picard scheme $\Pic_{Y/F}$, and let $\Pic^{\tau}_{Y/F}$ be as in \cite[Definition 9.6.8]{Kleiman}. The quotient $\Pic^{\tau}_{Y/F}/\Pic^0_{Y/F}$ is finite by \cite[Proposition 9.6.12]{Kleiman}. Note that $\Pic^{\tau}_{Y/F}$ parametrizes numerically trivial line bundles by \cite[Theorem 9.6.3, Exercise 9.6.11, and Proposition 9.6.12]{Kleiman}, together with the fact that numerical triviality is preserved under extension of the base field, which is checked by a descent argument.

Assume first that $\dim \Pic^0_{Y/F}=0$. Then $\Pic^0_{Y/F}$ is finite, and hence $\Pic^{\tau}_{Y/F}$ is also finite. Every numerically trivial line bundle on $Y$ is therefore torsion. Tensoring the natural surjection $\Pic(Y)\to N^1(Y)$ with $\Q$ gives the isomorphism in (1).

Assume now that $F$ is uncountable and algebraically closed, that the map in (1) is an isomorphism, and that $\dim \Pic^0_{Y/F}>0$. Let $A\coloneqq (\Pic^0_{Y/F})_{\rm red}$. Since $Y$ is normal, by \cite[Theorem 9.5.4]{Kleiman}, $A$ is an abelian variety of positive dimension. The group of points of $A$ rational over $F$, denoted by $A(F)$, is uncountable, whereas
\[
 A(F)_{\mathrm{tors}}=\bigcup_{n>0}A[n](F)
\]
is countable. Here $A[n]$ denotes the kernel of multiplication by $n$ on
$A$, and the group scheme $A[n]$ is finite for every $n>0$. Consequently,
$A(F)\subset\Pic^0_{Y/F}(F)$ contains a point of infinite order and it corresponds to a numerically trivial line bundle of infinite order on $Y$. The class of this line bundle is a nonzero element of $\Pic(Y)_{\Q}$ in the kernel of $\Pic(Y)_{\Q}\to N^1(Y)_{\Q}$, contradicting the injectivity.
\end{proof}

We next extend a finite collection of varieties, rational maps, line bundles, and sections from the generic fiber to an open subset of the base.

\begin{lemma}\label{lem:models-open}
Let $T$ be an integral variety with function field $K$, and let $Y_{1,\eta},\ldots,Y_{r,\eta}$ be normal projective $K$-varieties. Suppose that we are given finitely many rational maps among the $Y_{i,\eta}$, finitely many invertible sheaves on these varieties, and finitely many sections of these sheaves. Then, after replacing $T$ by a nonempty open subset $U$, there exist projective morphisms $Y_i\to U$ with $Y_i$ normal and generic fibers $Y_{i,\eta}$, together with extensions of all the given rational maps, invertible sheaves, and sections.

If a normal projective model of some $Y_{i,\eta}$ is already defined over $T$, then we may use its restriction to $U$.
\end{lemma}

\begin{proof}
Choose projective embeddings $Y_{i,\eta}\hookrightarrow\mathbb P^{N_i}_K$. Let $\overline Y_i$ be the closure of $Y_{i,\eta}$ in $\mathbb P^{N_i}_T$, equipped with the reduced induced structure. The scheme $\overline Y_i$ is integral and projective over $T$. Since $\overline Y_i$ is of finite type over a field, it is excellent. Let $Y_i$ be the normalization of $\overline Y_i$. The morphism $Y_i\to\overline Y_i$ is finite, and hence $Y_i$ is projective over $T$. Since $Y_{i,\eta}$ is normal, the generic fiber of $Y_i\to T$ is $Y_{i,\eta}$.

For each rational map $\phi_{\eta}\colon Y_{i,\eta}\dashrightarrow Y_{j,\eta}$, let $\Gamma_{\eta}\subset Y_{i,\eta}\times_K Y_{j,\eta}$ be the closure of its graph, and let $\Gamma\subset Y_i\times_T Y_j$ be its closure over $T$. The projection $\Gamma\longrightarrow Y_i$ is birational, while the morphism $\Gamma\longrightarrow Y_j$ induces a rational map $\phi\colon Y_i\dashrightarrow Y_j$ whose restriction to the generic fiber is $\phi_{\eta}$. The invertible sheaves extend after shrinking $T$ by \cite[Lemma 32.10.3(2), Tag 0B8W]{Stacks}. Let $L_{\eta}$ be one of the given invertible sheaves on $Y_{i,\eta}$, and let $L$ be its extension to $Y_i$. Each given section of $L_{\eta}$ corresponds to an $\OO_{Y_{i,\eta}}$-module homomorphism $\OO_{Y_{i,\eta}}\to L_{\eta}$. Since $\OO_{Y_i}$ and $L$ are of finite presentation, \cite[Lemma 32.10.2(2), Tag 01ZR]{Stacks} extends this homomorphism to $\OO_{Y_i}\to L$ after a further shrinking of $T$. The extended homomorphism defines a section of $L$ extending the given section. Since only finitely many sections are given, all of them extend over a common nonempty open subset of $T$.

If a normal projective model of some $Y_{i,\eta}$ is fixed a priori over $T$, then we use its restriction after replacing $T$ by a nonempty open subset.
\end{proof}

The next lemma shows that a birational map which is small on the generic fiber remains small after a suitable restriction of the base.

\begin{lemma}\label{lem:small-open}
Let $\phi_\eta\colon Y_\eta\dashrightarrow Z_\eta$ be a birational map between normal projective varieties over $K=k(T)$, and assume that $\phi_\eta$ is an isomorphism in codimension one. Let $\phi\colon Y\dashrightarrow Z$ be an extension over a nonempty open subset of $T$, where $Y$ and $Z$ are normal and projective over the base. After shrinking the base, the map $\phi$ is an isomorphism in codimension one.
\end{lemma}

\begin{proof}
By the valuative criterion of properness, the rational map $\phi$ is
defined at the generic point of every prime divisor of the normal variety
$Y$. Let $E\subset Y$ be a prime divisor contracted by $\phi$, and let
$\Gamma$ be the normalization of the closure of the graph of $\phi$, with
projections $p\colon\Gamma\to Y$ and $q\colon\Gamma\to Z$. The strict transform $\widetilde E\subset \Gamma$ of $E$ is not contracted by $p$ and is contracted by $q$. If $E$ dominates $T$, then $E_\eta$ and $\widetilde E_\eta$ are prime divisors, and $\phi_\eta$ contracts $E_\eta$. This contradicts the assumption that $\phi_\eta$ is an isomorphism in codimension one. Hence, every prime divisor of $Y$ contracted by $\phi$ is vertical. Only finitely many such divisors occur, since their strict transforms are among the finitely many prime divisors contracted by $q$.

The proper morphism $Y\to T$ maps each vertical prime divisor contracted by $\phi$ to a proper closed subset of $T$. After removing the union of these images from $T$, the restricted map $\phi$ contracts no prime divisor. After a further shrinking of the base, the same argument for $\phi^{-1}$ shows that neither $\phi$ nor $\phi^{-1}$ contracts a prime divisor. Since $Y$ and $Z$ are normal, $\phi$ is an isomorphism in codimension one.
\end{proof}

The following lemma allows us to assume, after shrinking the base, that the base is smooth and all fibers of the projective models are geometrically integral.

\begin{lemma}[{\protect\cite[Lemma 33.25.8, Tag 0B8X, Proposition 29.28.1, Tag 052A, Lemma 37.27.5, Tag 0559, and Lemma 37.26.7, Tag 0C0E]{Stacks}}]\label{lem:open-integral}
Let $k$ be a perfect field, let $T$ be an integral $k$-variety, and let $g\colon Y\to T$ be a projective morphism whose geometric generic fiber is integral. After replacing $T$ by a nonempty open subset, the variety $T$ is smooth, the morphism $g$ is flat, and every geometric fiber of $g$ is integral.
\end{lemma}

\subsection{Countable fields of definition}\label{subsec:countable}
The next two lemmas relate a sufficiently general $k$-point to the geometric generic point of a model over a countable field.

\begin{lemma}\label{lem:countable-model}
Let $k$ be an algebraically closed field, let
$f\colon X\to T$ be a projective morphism between normal integral
$k$-varieties, and let $S\subset T(k)$ be a subset not contained in any
countable union of proper closed subsets of $T$. Then there exist a
countable algebraically closed subfield $k_0\subset k$, a projective
morphism $f_0\colon X_0\to T_0$ between normal integral $k_0$-varieties
with $T_0$ geometrically integral, and a point $s\in S$ such that
$(f_0)_k\simeq f$ and the morphism $\Spec k\to T_0$ determined by $s$ has
image equal to the generic point of $T_0$.
\end{lemma}

\begin{proof}
By descent for schemes and morphisms of finite presentation, there is a finitely generated subfield
$k_1\subset k$ and a projective morphism $f_1\colon X_1\to T_1$ whose base
change to $k$ is $f$, see \cite[Lemma 32.10.1, Tag 01ZM]{Stacks}. Let $k_0$
be the algebraic closure of $k_1$ in $k$, and let
$f_0\colon X_0\to T_0$ be the base change $(f_1)_{k_0}$. The field $k_0$ is countable
and algebraically closed. Normality descends under faithfully flat field
extensions by \cite[Lemma 10.164.3, Tag 033G]{Stacks}. The varieties $X_0$
and $T_0$ are therefore normal, and they are integral since their base
changes to $k$ are integral. Since $k_0$ is algebraically closed, the
integral $k_0$-variety $T_0$ is geometrically integral.

Since $T_0$ is of finite type over the countable field $k_0$, the variety $T_0$ has only countably many proper closed subsets. The base changes of these closed subsets to $k$ form a countable collection of proper closed subsets of $T$. Choose $s\in S$ outside their union. If the morphism $\Spec k\to T_0$ determined by $s$ does not map to the generic point, then the closure of its image is a proper closed subset of $T_0$ whose base change belongs to the collection excluded by the choice of $s$, which is a contradiction.
\end{proof}

\begin{lemma}\label{lem:field-comparison}
Let $k$ be an uncountable algebraically closed field, let $k_0\subset k$ be a countable algebraically closed subfield, and let $T_0$ be a geometrically integral $k_0$-variety. Let $T\coloneqq T_0\times_{k_0}k$ and $K_0\coloneqq k_0(T_0)$. If $s\in T(k)$ maps to the generic point of $T_0$, then there is a $K_0$-isomorphism $\sigma\colon k\to \overline{k(T)}$ compatible with the embeddings determined by $s$ and by the generic point.
\end{lemma}

\begin{proof}
The point $s$ and the generic point of $T$ define embeddings $\iota_s\colon K_0\hookrightarrow k$ and $\iota_\eta\colon K_0\hookrightarrow\overline{k(T)}$. Put $\kappa\coloneqq \mathrm{trdeg}_{k_0}k$. Then $\kappa$ is uncountable, since an algebraic extension of a purely transcendental extension of a countable field with countable transcendence basis is countable.

The extension $K_0/k_0$ has finite transcendence degree, and hence $\mathrm{trdeg}_{(K_0,\iota_s)}k=\kappa$.
Since $T_0$ is geometrically integral, $K_0/k_0$ is regular. In particular,
$K_0$ and $k$ are linearly disjoint over $k_0$, and $\mathrm{trdeg}_{(K_0,\iota_\eta)}\overline{k(T)}=\kappa$.
Choose transcendence bases for the two extensions and fix a bijection
between the bases. The resulting $K_0$ isomorphism between the corresponding
purely transcendental extensions extends to an isomorphism between their
algebraic closures. Since $k$ and $\overline{k(T)}$ are algebraic closures
of these purely transcendental extensions, this extension gives the
required isomorphism $\sigma$. See also \cite[Lemma 2.13]{FLTZ}.
\end{proof}

\subsection{Cycles of rational curves}\label{subsec:cycles}
We first describe the Picard group and cohomology of a cycle of rational curves.

\begin{definition}[{\protect\cite[II, 1.1]{DR73}}]\label{def:ngon}
Let $\widetilde C=\P^1_{\Z}\times\Z/n$ be the disjoint sum of $n$ copies of
$\P^1_{\Z}$, indexed by $\Z/n$ ($n\geq1$). Gluing the $j$-th copy of
$\P^1_{\Z}$ to the $(j+1)$-st, by identifying the section $0$ of the $j$-th copy
with the section $\infty$ of the $(j+1)$-st, one obtains a curve $C$ of genus
one over $\Spec\Z$, with normalization $\widetilde C$. For every scheme $S$, the
\emph{standard $n$-gon over $S$} is the $S$-scheme $C\times_{\Z}S$.
\end{definition}

\begin{lemma}\label{lem:cycle-cohomology}
Let $S$ be a scheme, let $n\geq2$ be an integer, and let $A$ be the standard
$n$-gon over $S$ of Definition \ref{def:ngon}. For $j\in\Z/n$, let
$A_j\subset A$ be the image of the $j$-th copy of $\P^1_S$, and let
$q_j\in A(S)$ be the image of its section $0$. Put
$U_j\coloneqq A_j\setminus(q_{j-1}\cup q_j)$, so that
$U_j\simeq\mathbb G_{\mathrm m,S}$. For every $S$-scheme $T$, let $\Pic^{[0]}_{A/S}(T)$ consist of the classes in $\Pic_{A/S}(T)$ whose pullback to $A\times_S\bar t$ has multidegree zero for every geometric point $\bar t\to T$. Then the following statements hold.
\begin{enumerate}
 \item For every $S$-scheme $T$, there is a group isomorphism
 \[
  c_T\colon\Pic^{[0]}_{A/S}(T)\xrightarrow{\ \sim\ }\Gamma(T,\OO_T^\times),
 \]
 natural in $T$. In particular, $\Pic^{[0]}_{A/S}$ is representable by
 $\mathbb G_{\mathrm m,S}$.
 \item For every $i\in\Z/n$ and every section $p_0\in U_i(S)$, the Abel map
 \[
  \mathrm{ab}_{p_0}\colon U_i\longrightarrow\Pic^{[0]}_{A/S},
  \qquad
  p\longmapsto\OO_A(p_0-p),
 \]
 is an isomorphism of $S$-schemes.
 \item Suppose that $S=\Spec k$ for some field $k$. Then
 $\Pic^{[0]}_{A/k}=\Pic^0_{A/k}$. Furthermore, if $M$ is a line bundle of
 multidegree zero on $A$, then
 \[
  (h^0(A,M),h^1(A,M))=
  \begin{cases}
   (1,1),& M\simeq\OO_A,\\
   (0,0),& M\not\simeq\OO_A.
  \end{cases}
 \]
\end{enumerate}
\end{lemma}

\begin{proof}
Let $\nu\colon A^\nu=\coprod_{j\in\Z/n}A_j\to A$ be the base change of
$\widetilde C\to C$, which is the normalization on every geometric fiber. By the
construction of $C$, the sequence
\begin{equation}\label{eq:ngon-glueing}
 0\longrightarrow\OO_A\longrightarrow\nu_*\OO_{A^\nu}
 \xrightarrow{\ \delta\ }\bigoplus_{j\in\Z/n}q_{j*}\OO_S\longrightarrow0
\end{equation}
is exact, where $\delta$ sends a section $(a_j)_j$ to
$(a_j|_{q_j}-a_{j+1}|_{q_j})_j$; surjectivity of $\delta$ may be checked at the pairwise disjoint sections $q_1,\dots,q_n$, where it follows from the surjectivity of the restriction map $\OO_{A_j}\to q_{j*}\OO_S$.
Since $\bigoplus_jq_{j*}\OO_S$ is flat over $S$, the sequence
\eqref{eq:ngon-glueing} remains exact after arbitrary base change $T\to S$. The
subfunctor $\Pic^{[0]}_{A/S}$ is a subgroup functor, since the multidegree on a
geometric fiber is additive and is compatible with pullback along any morphism
$T'\to T$. We first compute $\Pic^{[0]}_{A/S}(T)$ for every $S$-scheme $T$.

Let $T$ be an arbitrary $S$-scheme, and let
$\pi\colon A_T\coloneqq A\times_ST\to T$ be the projection. Pushing the base
change of \eqref{eq:ngon-glueing} forward along $\pi$, and using
$(\pi_j)_*\OO_{A_j\times_ST}=\OO_T$ for every $j$, where
$\pi_j\colon A_j\times_ST\simeq\P^1_T\to T$ is the projection, we obtain a left
exact sequence
\[
0\longrightarrow\pi_*\OO_{A_T}\longrightarrow\bigoplus_{j\in\Z/n}\OO_T
 \xrightarrow{\ (a_j)_j\mapsto(a_j-a_{j+1})_j\ }\bigoplus_{j\in\Z/n}\OO_T,
\]
whose second map has the diagonal copy of $\OO_T$ as its kernel. Hence
$\pi_*\OO_{A_T}=\OO_T$. By \cite[Lemma 44.4.3, Tag 0B9N]{Stacks} applied to the
section $q_1\in A(S)$, we therefore have a natural isomorphism
$\Pic_{A/S}(T)\simeq\Pic(A_T)/\pi^*\Pic(T)$. In particular, every relative
Picard class has a representative on $A_T$, and two representatives differ by
the pullback of a line bundle on $T$.

By the construction of $A$, at a point $x$ of $q_j$ lying over $s\in S$ one has $\OO_{A,x}\simeq\OO_{A_j,x}\times_{\OO_{S,s}}\OO_{A_{j+1},x}$, and both restriction maps are surjective. As \eqref{eq:ngon-glueing} remains exact after arbitrary base change, the same description holds for $A_T$ over $T$. By \cite[Lemma 15.6.9, Tag 0D2J]{Stacks}, a line bundle on $A_T$ is equivalent to line bundles on the components $A_j\times_ST$, together with isomorphisms at each node between their restrictions to the two copies of $q_j\times_ST$.

Represent a class in $\Pic^{[0]}_{A/S}(T)$ by a line bundle $\mathcal M$ on $A_T$. Let $\mathcal M_j\coloneqq\mathcal M|_{A_j\times_ST}$, and let $\pi_j\colon A_j\times_ST\to T$ be the projection. For every geometric point $\bar t\to T$, the restriction $\mathcal{M}_j|_{A_j\times_S \bar t}$ has degree zero on $A_j\times_S \bar t\simeq \P^1_{\bar t}$, and hence is trivial. Consequently,
\[
h^0\bigl(A_j\times_S \bar t,\mathcal M_j|_{A_j\times_S\bar t}\bigr)=1,
\quad
h^1\bigl(A_j\times_S \bar t,\mathcal M_j|_{A_j\times_S\bar t}\bigr)=0,
\]
and the higher cohomology groups vanish. The morphism $\pi_j$ is proper and of finite presentation, and $\mathcal{M}_j$ is of finite presentation and flat over $T$. By \cite[Lemma 36.32.3, Tag 0B9S]{Stacks}, the sheaf $Q_j\coloneqq(\pi_j)_*\mathcal M_j$ is invertible on an open subset containing every geometric point of $T$, hence on all of $T$. For a morphism $v\colon T'\to T$, let $v_j\coloneqq \mathrm{id}_{A_j}\times v$, and let $\pi_j'\colon A_j\times_S T'\to T'$ be the projection. The natural map $v^{\ast}Q_j\to (\pi_j')_{\ast}v_j^{\ast}\mathcal{M}_j$ is an isomorphism by \cite[Lemma 36.30.4(2), Tag 0B91]{Stacks}. The evaluation map $\pi_j^{\ast}Q_j\to \mathcal{M}_j$ restricts on every geometric fiber to the evaluation isomorphism of a trivial line bundle on a projective line. Since $\pi_j^{\ast}Q_j$ and $\mathcal{M}_j$ are invertible, the evaluation map is an isomorphism on $A_j\times_S T$.

Under the evaluation isomorphisms $\pi_j^*Q_j\simeq\mathcal M_j$, the restrictions at the two preimages of $q_j\times_ST$ are $Q_j$ and $Q_{j+1}$. The gluing map at $q_j\times_ST$ therefore induces an isomorphism $\theta_j\colon Q_j\xrightarrow{\sim}Q_{j+1}$. The composition $\theta_n\circ\cdots\circ\theta_1$ is an automorphism of $Q_1$. Since $\mathcal{E}nd_{\OO_T}(Q_1)\simeq\OO_T$, there is a unique unit $\lambda\in\Gamma(T,\OO_T^\times)$ such that $\theta_n\circ\cdots\circ\theta_1=\lambda\,\mathrm{id}_{Q_1}$. An isomorphism between two representatives conjugates the corresponding compositions and does not change the scalar $\lambda$. Tensoring $\mathcal M$ with $\pi^*L$, for a line bundle $L$ on $T$, replaces $Q_j$ and $\theta_j$ by $Q_j\otimes L$ and $\theta_j\otimes\mathrm{id}_L$, respectively, and also leaves $\lambda$ unchanged. Hence, $\lambda$ depends only on the relative Picard class of $\mathcal M$.

Conversely, for any unit $\lambda\in\Gamma(T,\OO_T^\times)$, let $\mathcal M(\lambda)$ be the line bundle obtained by gluing the trivial line bundles on $A_j\times_ST$, using the identity at $q_1\times_ST,\ldots,q_{n-1}\times_ST$ and multiplication by $\lambda$ at $q_n\times_ST$. To recover the relative class of $\mathcal M$, put $\alpha_1\coloneqq\mathrm{id}_{Q_1}$ and $\alpha_j\coloneqq\theta_{j-1}\circ\cdots\circ\theta_1\colon Q_1\xrightarrow{\sim}Q_j$ for $2\leq j\leq n$. For $j<n$, one has $\alpha_{j+1}^{-1}\circ\theta_j\circ\alpha_j=\mathrm{id}_{Q_1}$, whereas $\theta_n\circ\alpha_n=\lambda\,\mathrm{id}_{Q_1}$. Thus, after using the isomorphisms $\alpha_j$ on the components, the gluing maps are the identity at the first $n-1$ nodes and multiplication by $\lambda$ at the last node. Consequently, $\mathcal M\simeq\pi^*Q_1\otimes\mathcal M(\lambda)$. The maps $[\mathcal M]\mapsto\lambda$ and $\lambda\mapsto[\mathcal M(\lambda)]$ are therefore mutually inverse. The pullback of each $\theta_j$ is the gluing map for the pulled back line bundle, by the base change isomorphisms for $Q_j$. Hence, both maps are natural in $T$. Moreover, gluing on each component gives $\mathcal M(\lambda)\otimes\mathcal M(\lambda') \simeq\mathcal M(\lambda\lambda')$. We obtain natural group isomorphisms
\[
 c_T\colon\Pic^{[0]}_{A/S}(T)\xrightarrow{\sim}\Gamma(T,\OO_T^\times),
 \qquad [\mathcal M]\longmapsto\lambda.
\]
Since $\mathbb G_{\mathrm m,S}(T)=\Gamma(T,\OO_T^\times)$ for every $S$-scheme $T$, the functor $\Pic^{[0]}_{A/S}$ is representable by $\mathbb G_{\mathrm m,S}$. This proves (1). 

In what follows, we write $c\colon\Pic^{[0]}_{A/S}\xrightarrow{\sim}\mathbb G_{\mathrm m,S}$ for the isomorphism induced by the map $c_T$. For the gluing constants, note that $Q_j\simeq Q_1$ for every $j$. Thus, $\mathcal M\otimes\pi^*Q_1^{-1}$ is trivial on every component $A_j\times_ST$ and represents the same relative Picard class as $\mathcal M$. After replacing $\mathcal M$ by $\mathcal M\otimes\pi^*Q_1^{-1}$, choose a basis $e_j$ on each component. If the gluing map at $q_j\times_ST$ sends $e_j$ to $\lambda_je_{j+1}$, then $c_T([\mathcal M])=\lambda_1\cdots\lambda_n$.

We next prove (2). By the construction of $A$, the component $A_i$ carries a coordinate $u_0$ with $u_0(q_i)=0$ and $u_0(q_{i-1})=\infty$. Since $p_0\in U_i(S)$, the value $a\coloneqq u_0(p_0)$ lies in $\Gamma(S,\OO_S^\times)$, and the coordinate $u\coloneqq u_0/a$ on $A_i$ satisfies $u(q_i)=0$, $u(q_{i-1})=\infty$, and $u(p_0)=1$. Let $u_1$ be the pullback of $u$ from $A_i$, and let $u_2$ be the pullback of $u|_{U_i}$ from $U_i$. Let $\Delta\subset A\times_SU_i$ be the graph of $U_i\hookrightarrow A$. The graph $\Delta$ and $p_0\times_SU_i$ are sections of $A\times_SU_i\to U_i$ contained in the relative smooth locus. Since $A\times_SU_i\to U_i$ has relative dimension one, both sections are relative effective Cartier divisors. The line bundle $\mathcal L\coloneqq \OO_{A\times_SU_i}(p_0\times_SU_i-\Delta)$ has multidegree zero on every geometric fiber and defines the Abel morphism $\mathrm{ab}_{p_0}\colon U_i\to\Pic^{[0]}_{A/S}$.

On $A_i\times_SU_i$, the rational function $h\coloneqq \frac{u_1-1}{u_1-u_2}$ has divisor $p_0\times_SU_i-\Delta$. Hence, $h^{-1}$ is a nowhere vanishing section of $\mathcal L|_{A_i\times_SU_i}$. For $j\neq i$, the divisor $p_0\times_SU_i-\Delta$ is disjoint from $A_j\times_SU_i$, and the restriction of $\mathcal L$ to $A_j\times_SU_i$ has the constant section $1$ as a basis. Identifying $q_j\times_SU_i$ with $U_i$, the restrictions of $h^{-1}$ to $q_i\times_SU_i$ and $q_{i-1}\times_SU_i$ are $u$ and $1$, respectively. At $q_i\times_SU_i$, the basis on $A_i\times_SU_i$ has value $u$ and the basis on $A_{i+1}\times_SU_i$ is $1$. The gluing map from $A_i\times_SU_i$ to $A_{i+1}\times_SU_i$ therefore has coefficient $u$. Every other gluing coefficient is $1$, and the product formula for $c$ gives $c\circ\mathrm{ab}_{p_0}=u$. Since $u\colon U_i\to\mathbb G_{\mathrm m,S}$ and $c$ are isomorphisms, $\mathrm{ab}_{p_0}$ is an isomorphism of $S$-schemes. This proves (2).

Finally, assume that $S=\Spec k$ for a field $k$. Since the multidegree is locally constant on $\Pic_{A/k}$ and is zero at the identity, $\Pic^0_{A/k}\subseteq\Pic^{[0]}_{A/k}$. Conversely, $\Pic^{[0]}_{A/k}\simeq\mathbb G_{\mathrm m,k}$ by (1), which is connected and contains the identity, and hence $\Pic^{[0]}_{A/k}\subseteq\Pic^0_{A/k}$. Therefore, $\Pic^{[0]}_{A/k}=\Pic^0_{A/k}$.

It remains to compute the cohomology. Let $M$ be a line bundle of multidegree zero on $A$. Choose a trivialization of $M|_{A_j}$ on every component, and write the gluing map at $q_j$ as $e_j\mapsto\lambda_je_{j+1}$, with $\lambda_j\in k^\times$. A global section of $M$ is a tuple $(a_1,\ldots,a_n)\in k^n$ satisfying $a_{j+1}=\lambda_ja_j$ for every $j$. The first $n-1$ equations give $a_j=(\lambda_1\cdots\lambda_{j-1})a_1$ for $2\leq j\leq n$. The relation at the last node is $(1-\lambda_1\cdots\lambda_n)a_1=0$. The product $\lambda_1\cdots\lambda_n$ equals $1$ precisely when $M\simeq\OO_A$, by the construction of $c$ and the equality $\Pic(\Spec k)=0$. Therefore, $h^0(A,M)=1$ when $M\simeq\OO_A$, and $h^0(A,M)=0$ otherwise. Tensoring \eqref{eq:ngon-glueing} with $M$, and choosing bases in the fibers of $M$ at the nodes, gives
\[
 0\longrightarrow M\longrightarrow\nu_*\nu^*M
 \longrightarrow\bigoplus_{j\in\Z/n} k(q_j)\longrightarrow0.
\]
Since $\nu$ is finite and $M|_{A_j}\simeq\OO_{A_j}$ for every $j$, taking Euler characteristics yields $\chi(A,M)=\sum_{j\in\Z/n}\chi(A_j,\OO_{A_j})-n=0$. Hence, $h^1(A,M)=h^0(A,M)$, which completes the proof.
\end{proof}

\subsection{Surface contractions and semiampleness}\label{subsec:surface}
The following contraction theorem and basepoint free theorem for smooth projective surfaces are well known.

\begin{proposition}[{\protect\cite[Theorems 3.3 and 3.7]{KM},
\protect\cite[Theorems 1.3, 3.13, and 3.21]{Tanaka}}]\label{prop:surface-mmp}
Let $X$ be a smooth projective surface over an algebraically closed field.
The cone theorem holds for $X$, and every $K_X$-negative extremal ray admits
a contraction. Moreover, if $D$ is a nef Cartier divisor such that
$D-K_X$ is nef and big, then $D$ is semiample.
\end{proposition}

\begin{remark}\label{rem:ZD}
    We use Zariski decompositions on surfaces over algebraically closed fields of arbitrary characteristic. Let $X$ be a smooth projective surface over such a field, and let $D$ be a pseudoeffective $\R$-divisor on $X$. By \cite[Theorem 3.1]{Eno24}, there is a unique decomposition $D=P+N$ such that $P$ is a nef $\R$-divisor, $N$ is zero or an effective $\R$-divisor whose prime components have negative definite intersection matrix, and $P\cdot C=0$ for every prime component $C$ of $N$.
\end{remark}

\section{Main results}\label{sec:main}
In this section, we construct the family which gives the two counterexamples and then prove Theorem \ref{thm:main}.

\subsection{A family of rational surfaces with an anticanonical cycle}\label{sec:family}

\begin{construction}\label{con:family}
Let $L_1,L_2,L_3\subset\mathbb P^2_{\mathbb Z}$ be the coordinate lines,
with indices taken modulo $3$. Blow-up the vertices $v_i=L_i\cap L_{i+1}$ and denote the exceptional curves by $E_i$. Then blow-up the point $E_i\cap L_{i+1}$, and denote the exceptional curve by $F_i$. Let $S$ be the resulting smooth toric surface, and denote the strict transforms of $L_i$ and $E_i$ on $S$ by $L_i'$ and $E_i'$, respectively.

The toric boundary of $S$ is the cycle
\[
 D_S=L_1'+E_1'+F_1+L_2'+E_2'+F_2+L_3'+E_3'+F_3,
\]
where the components occur in the cyclic order. Their
self-intersection numbers are $(L_i')^2=(E_i')^2=-2$ and $F_i^2=-1$.
The cycle $D_S$ is isomorphic to the standard $9$-gon over $\Z$ of Definition \ref{def:ngon}; fix such an identification, matching the cyclic order above.
For each $i$, set $F_i^\circ\coloneqq F_i\setminus(E_i'\cup L_{i+1}')$. Then $F_i^\circ\simeq\mathbb G_{\mathrm m,\Z}$.

Choose the centers of the last blow-up as follows. The divisor $D_S$ has degree one on each $F_i$ and degree zero on each $L_i'$ and $E_i'$. Consequently, the line bundle
\[
 \OO_{D_S}\bigl(D_S-r_1-r_2-r_3\bigr)
\]
has multidegree zero for any sections $r_i\in F_i^\circ(\Z)$. Choose $r_2\in F_2^\circ(\Z)$, $r_3\in F_3^\circ(\Z)$, and a base point $r_1^0\in F_1^\circ(\Z)$. By Lemma \ref{lem:cycle-cohomology}, the Abel map with base point $r_1^0$ is an isomorphism. Hence, there is a unique $r_1\in F_1^\circ(\Z)$ such that
\[
\OO_{D_S}(D_S-r_1-r_2-r_3)\simeq\OO_{D_S},
\]
where we use that $\Pic(\Spec\Z)=0$. For each $i$, let $z_i\colon F_i^\circ\xrightarrow{\sim}\mathbb G_{\mathrm m,\Z}$ be the coordinate characterized by
\[
 [\OO_{D_S}(r_i-p)]=z_i(p)
 \quad\text{in}\quad
 \Pic^{[0]}_{D_S/\Z}\simeq\mathbb G_{\mathrm m,\Z}.
\]
Put $B\coloneqq\Spec\Z[q,q^{-1}]$. In $S\times_{\Z}B$, let $p_1,p_2,p_3$ be the sections defined by
\[
 z_1(p_1)=q,
 \qquad
 p_2=r_2,
 \qquad
 p_3=r_3.
\]
The sections $p_1,p_2,p_3$ lie on distinct boundary components and are therefore disjoint. Blowing up their union gives a smooth projective morphism $\Pi\colon\mathfrak X\longrightarrow B$ with exceptional divisors $G_1,G_2,G_3$.

Let $\mathfrak A$ be the strict transform of $D_S\times_{\Z}B$. Since
$D_S=-K_{S/\Z}$ and each center is a smooth point of $D_S$, the canonical
divisor formula gives $\mathfrak A=-K_{\mathfrak X/B}$.
Let $k$ be an algebraically closed field, and let $q\in k^{\times}$. Write $X_q$ for the corresponding geometric fiber, and denote the strict transforms on $X_q$ by 
$\widetilde L_i,\widetilde E_i,\widetilde F_i$. The divisor
\[
 A_q=\sum_{i=1}^3
 (\widetilde L_i+\widetilde E_i+\widetilde F_i)
\]
is a reduced cycle of nine smooth rational curves. Every component $C$ of $A_q$ satisfies $C^2=-2$. The only nonzero intersections between distinct components are
\[
 \widetilde L_i\cdot\widetilde E_i
 =\widetilde E_i\cdot\widetilde F_i
 =\widetilde F_i\cdot\widetilde L_{i+1}=1.
\]
Thus, $A_q$ is a cycle of type $I_9$. In particular, $A_q^2=0$ and
$A_q\cdot C=0$ for every component $C$ of $A_q$. If $C$ is any other
integral curve, then $A_q\cdot C\geq0$ since $A_q$ is effective. Hence,
$A_q$ is nef.

Let $p_i(q)\in F_{i,k}^{\circ}(k)$ denote the value of the section $p_i$ at $q$, and let $\pi_q\colon X_q\to S_k$ be the blow-up. The restriction of $\pi_q$ induces an isomorphism
$A_q\xrightarrow{\sim}D_{S,k}$. Under this isomorphism, the formula for the
strict transform gives
\[
 \OO_{A_q}(A_q)
 \simeq
 \OO_{D_{S,k}}\bigl(D_{S,k}-p_1(q)-p_2(q)-p_3(q)\bigr).
\]
Indeed,
$\OO_{X_q}(A_q)=\pi_q^*\OO_{S_k}(D_{S,k})\otimes
\OO_{X_q}(-G_{1,q}-G_{2,q}-G_{3,q})$, and restricting to $A_q$ gives the isomorphism. By the choice of the reference points $r_i$ and the
coordinates $z_i$,
\[
 \begin{aligned}
 [\OO_{A_q}(A_q)]
 &=[\OO_{D_{S,k}}(D_{S,k}-r_1-r_2-r_3)]
   \prod_{i=1}^3[\OO_{D_{S,k}}(r_i-p_i(q))]\\
 &=z_1(p_1(q))z_2(p_2(q))z_3(p_3(q))=q
 \end{aligned}
\]
in $\Pic^{[0]}_{A_q/k}\simeq\mathbb G_{\mathrm m,k}$.
\end{construction}

Thus, $q=[\OO_{A_q}(A_q)]$ in $\Pic^{[0]}_{A_q/k}=\Pic^0_{A_q/k}\simeq\mathbb G_{\mathrm m,k}$. The next proposition gives a description of the negative curves, the Mori cone, and the semiampleness of $A_q$.

\begin{proposition}\label{prop:fiberwise}
Let $k$ be an algebraically closed field, and let
$q\in k^\times$.  For the fiber $(X_q,A_q)$ of
Construction \ref{con:family}, the following statements hold.
\begin{enumerate}[(1)]
 \item The negative curves are the nine components of $A_q$ and the three
 exceptional curves $G_{1,q},G_{2,q},G_{3,q}$.
 \item These twelve curves generate $\overline{\NE}(X_q)$.  In particular,
 the Mori and nef cones are rational polyhedral and independent of $q$ after identifying the N\'{e}ron--Severi spaces by the fixed basis of total transforms.
 \item The divisor $A_q$ is semiample if and only if $q$ has finite
 multiplicative order.
 \item If $q$ has exact order $m$, then $m$ is the least positive integer
 $n$ for which $\dim|nA_q|>0$. Moreover, the pencil $|mA_q|$ has no base points and has connected fibers of arithmetic
 genus one.
 \item The surface $X_q$ is a Mori dream space if and only if $q$ has finite
 multiplicative order.
\end{enumerate}
\end{proposition}

\begin{proof}
By the construction, we have $\rho(X_q)=10$ and the classes of total transforms $H,E_i,F_i,G_i$, for $1\leq i\leq 3$, form a basis of $\Pic(X_q)$. Denote the components of $A_q$ in cyclic order by $A_{1,q},\dots,A_{9,q}$. The intersection matrix of the $I_9$-cycle has rank eight. Suppose that $\sum_{j=1}^9 a_jA_{j,q}=0$ is a linear relation among the components of $A_q$. Intersecting successively with the components of the cycle shows that $a_1=\cdots=a_9$. Hence,
\[
\sum_{j=1}^9 a_jA_{j,q}=a_1A_q.
\]
Intersecting this equality with $G_{1,q}$ gives $0=a_1(A_q\cdot G_{1,q})=a_1$, since $A_q\cdot G_{1,q}=1$. Therefore, all $a_j$ vanish. Hence, the classes of the nine components of $A_q$ are linearly independent and span $A_q^\perp\subset N^1(X_q)_{\R}$.

For every integral curve $C$, adjunction gives $C^2=2p_a(C)-2+A_q\cdot C$. Since $A_q$ is nef, a negative curve is either a $(-2)$-curve orthogonal to $A_q$ or a $(-1)$-curve meeting $A_q$ once. Suppose that $C$ is a $(-2)$-curve which is not a component of $A_q$. Since $C$ is distinct from every component $A_{j,q}$ of $A_q$, one has $C\cdot A_{j,q}\geq 0$ for every $j$. On the other hand,
\[
0=A_q\cdot C=\sum_{j=1}^9 A_{j,q}\cdot C.
\]
Hence, $A_{j,q}\cdot C=0$ for every $j$.
Therefore, $C$ is orthogonal to every component of $A_q$. The numerical
class of $C$ is then proportional to $A_q$, contradicting $C^2<0$.

To determine the $(-1)$-curves, write divisor classes in the basis of total transforms $H,E_i,F_i,G_i$. We use cyclic indices, with $E_0=E_3$ and $F_0=F_3$. Then the classes of the components of $A_q$ are
\[
\widetilde L_i=H-E_{i-1}-E_i-F_{i-1},\qquad \widetilde E_i=E_i-F_i,\qquad \widetilde F_i=F_i-G_i, \qquad (1\leq i\leq 3).
\]
A $(-1)$-curve not contained in $A_q$ meets exactly one component of
$A_q$, with intersection number one. Write its numerical class as
\[
 D=dH-\sum_{i=1}^3e_iE_i-\sum_{i=1}^3f_iF_i
   -\sum_{i=1}^3g_iG_i.
\]
Cyclic symmetry reduces the calculation to the three types of components of $A_q$. Solving the nine intersection equations in these three cases
gives
\[
\begin{array}{c|c|c}
\text{component met} & \text{solution} & \text{consequence}\\
\hline\rule{0pt}{2.7ex}
\widetilde F_1 & D=G_1+tA_q,\quad D^2=-1+2t & t=0\\
\widetilde E_1 & (g_1,g_2,g_3)=(t-\frac13,t+\frac13,t)
  & \text{no integral class}\\
\widetilde L_1 & (g_1,g_2,g_3)=(t+\frac13,t+\frac23,t)
  & \text{no integral class}.
\end{array}
\]
By cyclic symmetry again, every $(-1)$-curve is numerically equivalent to one of
$G_{1,q},G_{2,q},G_{3,q}$. An irreducible curve distinct from $G_{i,q}$
cannot have the same numerical class, since its intersection with
$G_{i,q}$ would then be $-1$. Hence, the only $(-1)$-curves are
$G_{1,q},G_{2,q},G_{3,q}$. Since the computation depends only on the integral intersection matrix in the basis of total transforms, it remains valid after arbitrary base change from $\Z$. In characteristic zero, the resulting
configuration is the $H\widetilde A_8$ configuration of
\cite[Example 1.4.1]{Nikulin}. This proves (1).

Since $X_q$ is a smooth surface, we identify the numerical divisor and curve spaces via the intersection pairing. Let $D$ be a pseudoeffective $\R$-divisor on $X_q$ satisfying $A_q\cdot D=0$, and write its Zariski decomposition as $D=P+N$, whose existence in arbitrary characteristic was recalled in Remark \ref{rem:ZD}. The nefness of $A_q$ and $P$, together with the effectivity of $N$, gives $A_q\cdot P=A_q\cdot N=0$. Let $C$ be a prime component of $N$. The negative definiteness of the intersection matrix of the components of $N$ gives $C^2<0$. Since $A_q$ is nef and $A_q\cdot N=0$, we have $A_q\cdot C=0$. By (1), $C$ is either a component of $A_q$ or one of the curves $G_{1,q},G_{2,q},G_{3,q}$. The equality $A_q\cdot G_{i,q}=1$ for every $i$ excludes these three curves. Hence, every prime component of $N$ is a component of $A_q$. Since $P$ is nef, we have $P^2\geq 0$. The Hodge index theorem and the equalities $A_q^2=A_q\cdot P=0$ give $P\equiv cA_q$ for some $c\geq 0$. Since $A_q$ is the sum of its components, we obtain
\[
 \overline{\Eff}(X_q)\cap A_q^\perp
 =\mathrm{Cone}(A_{1,q},\ldots,A_{9,q}).
\]
Every extremal ray of $\overline{\NE}(X_q)$ not contained in $A_q^\perp$ is
$K_{X_q}$-negative. By Proposition \ref{prop:surface-mmp}, each such ray
admits a contraction. A contraction to a point would give
$\rho(X_q)=1$, while a contraction onto a projective curve would give
$\rho(X_q)=2$. Both conclusions contradict $\rho(X_q)=10$. Hence, every
such contraction is birational, and its extremal ray is generated by one
of the three $(-1)$-curves. This proves (2).

Set $\eta_q\coloneqq\OO_{A_q}(A_q)$. Since $X_q$ is rational,
$H^1(X_q,\OO_{X_q})=0$. By Lemma \ref{lem:cycle-cohomology}, both
$H^0(A_q,\eta_q^n)$ and $H^1(A_q,\eta_q^n)$ vanish whenever
$\eta_q^n\not\simeq\OO_{A_q}$. If $\eta_q^n\simeq\OO_{A_q}$, then both
cohomology groups are of dimension one.

For every integer $n>0$, consider the following exact sequences
\[
 0\longrightarrow\OO_{X_q}((n-1)A_q)
 \longrightarrow\OO_{X_q}(nA_q)
 \longrightarrow\eta_q^n\longrightarrow0.
\]
If $q$ has infinite order, then induction gives $h^0(X_q,nA_q)=1$ for every
$n>0$, and $A_q$ is not semiample. If $A_q$ is semiample, then there exists a positive integer $n$ such that $\mathcal O_{X_q}(nA_q)$ is globally generated. The restriction $\mathcal O_{A_q}(nA_q)=\eta_q^n$ is also globally generated. Since it has multidegree zero on the cycle $A_q$, it is trivial. Hence, $\eta_q$ is torsion, and therefore $q$ has finite multiplicative order.

Conversely, assume that $q$ has exact order $m$. For $0<n<m$, the same exact
sequences give $h^0(X_q,nA_q)=1$ and $h^1(X_q,nA_q)=0$. For $n=m$, the
vanishing of $H^1(X_q,(m-1)A_q)$ implies that the map $H^0(X_q,mA_q)\to H^0(A_q,\eta_q^m)=k$ is surjective, and therefore $h^0(X_q,mA_q)=2$. Choose $s_1\in H^0(X_q,mA_q)$ whose restriction to $A_q$ is a nonzero constant, and let $s_0$ be the section with divisor $mA_q$. The divisor of zeroes of $s_1$ is disjoint from $A_q$, whereas $s_0$ is nonzero away from $A_q$. Thus, $s_0$ and $s_1$ have no common zero, and the pencil they generate is basepoint free. Write the Stein factorization of the morphism defined by the pencil as
\[
X_q\xrightarrow{g}C\xrightarrow{h}\mathbb P^1.
\]
Since $X_q$ is rational and $g_*\OO_{X_q}=\OO_C$, one has $H^1(C,\OO_C)=0$. Thus, $C\simeq\mathbb P^1$. If $d=\deg h$, then
$h^*\OO_{\mathbb P^1}(1)\simeq\OO_{\mathbb P^1}(d)$, and
\[
 h^0(X_q,mA_q)=h^0(C,h^*\OO_{\mathbb P^1}(1))=d+1.
\]
Since $h^0(X_q,mA_q)=2$, we have $d=1$. Hence, the morphism defined by
the pencil has connected fibers. Adjunction gives arithmetic genus one for
every fiber. This proves
(3) and (4).

It remains to prove (5). Let $D$ be a nef Cartier divisor. If
$D\cdot A_q>0$, then $D-K_{X_q}=D+A_q$ is nef and big, since
$(D+A_q)^2\geq2D\cdot A_q>0$. Proposition \ref{prop:surface-mmp}
therefore implies that $D$ is semiample. If $D\cdot A_q=0$, then the Hodge
index theorem gives $D\equiv cA_q$ for some $c\geq0$. Since
$A_q\cdot G_{1,q}=1$, one has
$c=D\cdot G_{1,q}\in\mathbb Z_{\geq0}$. The basis of total transforms
identifies $\Pic(X_q)$ with a free abelian group of rank ten, and the
intersection pairing is nondegenerate. Consequently, numerical and linear
equivalence coincide on $X_q$, and $D\sim cA_q$. Hence, every nef Cartier
divisor is semiample if and only if $A_q$ is semiample.

The cone $\Nef(X_q)$ is rational polyhedral by (2), and
$\Pic(X_q)_{\Q}=N^1(X_q)_{\Q}$. Every movable divisor on a smooth
projective surface is nef. If $A_q$ is semiample, then every nef Cartier divisor is semiample, and hence some positive multiple is movable. Therefore, $\Nef(X_q)\subseteq\Mov(X_q)$. Since the reverse inclusion holds on every smooth projective surface, we have $\Mov(X_q)=\Nef(X_q)$, and $X_q$ is a Mori dream space. 

Conversely, if $X_q$ is a Mori dream space, then $A_q$ is semiample by Remark \ref{rem:nef-semiample}, since $A_q$ is nef. Together with (3), this proves (5).
\end{proof}

\subsection{The counterexamples}\label{sec:counterexamples}
We now apply Proposition \ref{prop:fiberwise} in characteristic zero and
in positive characteristic.

\subsubsection{Characteristic zero}

\begin{proposition}\label{prop:complex-counterexample}
After base change of Construction \ref{con:family} to $\mathbb C$, we have
\[
 \{q\in\mathbb C^\times\mid X_q\text{ is a Mori dream space}\}
 =\mu_\infty.
\]
Both $\mu_\infty$ and $\mathbb C^\times\setminus\mu_\infty$ are Zariski
dense, and $\mu_\infty$ is not constructible. For every nonempty open
subset $U\subset\mathbb G_{\mathrm m,\mathbb C}$, the morphism
$\mathfrak X_U\to U$ is not a Mori dream morphism. After any dominant
generically finite base change from a normal integral curve, the resulting
family is not a Mori dream morphism over any nonempty open subset of the
new base.
\end{proposition}

\begin{proof}
By Proposition \ref{prop:fiberwise}, the surface $X_q$ is a Mori dream space precisely when $q$ is a root of unity, which proves the equality. Every infinite subset of the irreducible curve
$\mathbb G_{\mathrm m,\mathbb C}$ is Zariski dense. Hence, both
$\mu_\infty$ and its complement are dense. A dense constructible subset
of an irreducible curve contains a nonempty open subset. Since
$\mathbb C^\times\setminus\mu_\infty$ meets every nonempty open
subset, $\mu_\infty$ is not constructible.

Let $\tau\colon C\to U$ be a dominant generically finite morphism from a normal integral curve to a nonempty open subset $U\subset\mathbb G_{\mathrm m,\mathbb C}$, and let $C^\circ\subset C$ be a nonempty open subset. The function $\tau^{\ast}q\in \C(C)^{\times}$ is nonconstant and therefore has infinite
multiplicative order. The divisor $\tau^*\mathfrak A$ is relatively nef,
whereas Proposition \ref{prop:fiberwise} shows that its restriction to the
geometric generic fiber of $C$ is not semiample. If the family obtained by
base change to $C^\circ$ is a Mori dream morphism, then $\tau^*\mathfrak A$ is relatively semiample by Remark \ref{rem:nef-semiample}, a contradiction. 

Taking $C=C^\circ=U$ and $\tau=\mathrm{id}_U$ proves the assertion for $\mathfrak X_U\to U$. The same argument applies to an arbitrary dominant generically finite base change.
\end{proof}

Proposition \ref{prop:complex-counterexample} gives a negative answer to \cite[Question 3.2]{CLZ}. 

\subsubsection{Positive characteristic}

The family obtained by reduction modulo $p$ has the additional property that
every geometric fiber over a closed point is a Mori dream surface, whereas
its geometric generic fiber is not a Mori dream space.

\begin{proposition}\label{prop:finite-field-counterexample}
Let $p$ be a prime, and consider
\[
 \Pi_p\colon\mathfrak X_p\longrightarrow\mathbb G_{\mathrm m,\mathbb F_p}.
\]
Every geometric fiber of $\Pi_p$ over a closed point is a Mori dream
surface, whereas the geometric generic fiber is not a Mori dream space. For
every nonempty open subset $U\subset\mathbb G_{\mathrm m,\mathbb F_p}$, the
morphism $\mathfrak X_{p,U}\to U$ is not a Mori dream morphism. After any
dominant generically finite base change from a normal integral curve,
including a purely inseparable base change, the resulting family is not a
Mori dream morphism over any nonempty open subset of the new base.
\end{proposition}

\begin{proof}
Let $b$ be a closed point of $\mathbb G_{\mathrm m,\mathbb F_p}$.
Then $\kappa(b)=\mathbb F_{p^r}$ for some $r>0$, and
$q(b)\in\kappa(b)^\times$ has finite multiplicative order.
Proposition \ref{prop:fiberwise} shows that the geometric fiber
over $b$ is a Mori dream surface. On the other hand, $q$ has infinite multiplicative order in $\overline{\mathbb F_p(q)}^\times$, and the geometric generic fiber is not a Mori dream space. If the restriction over a nonempty open subset is a Mori dream morphism,
then the relatively nef divisor $\mathfrak A_p$ is relatively
semiample by Remark \ref{rem:nef-semiample}. Relative semiampleness would imply
semiampleness on the geometric generic fiber, which is a contradiction.

Let $\tau\colon C\to U$ be a dominant generically finite morphism from a normal
integral curve to a nonempty open subset
$U\subset\mathbb G_{\mathrm m,\mathbb F_p}$, and let
$C^\circ\subset C$ be a nonempty open subset. The function
$\tau^*q\in\mathbb F_p(C)^\times$ is nonconstant and therefore has infinite
multiplicative order. Proposition \ref{prop:fiberwise} shows that the
pullback of $\mathfrak A_p$ is not semiample on the geometric generic fiber
of $C$. If the family over $C^\circ$ is a Mori dream morphism, then the relatively nef
pullback of $\mathfrak A_p$ is relatively semiample by Remark \ref{rem:nef-semiample}. Relative semiampleness implies semiampleness on the geometric generic fiber, a contradiction. The same proof applies when $\tau$ is purely inseparable.
\end{proof}

\subsection{The geometric generic fiber}\label{sec:generic}

We use the following comparison between relative divisor classes and divisor classes on the generic fiber.

\begin{lemma}\label{lem:qfactor-open}
Let $k$ be a perfect field and let $g\colon Y\to V$ be a projective dominant
morphism between normal integral $k$-varieties.  Suppose that the geometric
generic fiber is integral and that $Y_\eta$ is $\Q$-factorial and $\Pic(Y_\eta)_{\Q}$ is finite dimensional.  After replacing $V$ by a nonempty
open subset, the following statements hold.
\begin{enumerate}[(1)]
 \item The variety $V$ is smooth, the morphism $g$ is flat, and every
 geometric fiber is integral.
 \item The variety $Y$ is $\Q$-factorial.
 \item Restriction induces an isomorphism
 \[
  \Pic(Y/V)_{\Q}\xrightarrow{\sim}\Pic(Y_\eta)_{\Q}.
 \]
\end{enumerate}
If $\Pic(Y_\eta)_{\Q}=N^1(Y_\eta)_{\Q}$, then
$\Pic(Y/V)_{\Q}=N^1(Y/V)_{\Q}$.
\end{lemma}

\begin{proof}
By Lemma \ref{lem:open-integral}, after shrinking $V$, we may assume that $V$ is smooth, $g$ is flat, and every geometric fiber is integral.

Choose line bundles $L_{1,\eta},\ldots,L_{\rho,\eta}$ whose classes form a basis of $\Pic(Y_\eta)_{\Q}$. By Lemma \ref{lem:models-open}, after a further shrinking, each $L_{i,\eta}$ extends to a line bundle $L_i$ on $Y$.

Let $E\subset Y$ be a vertical prime divisor, and set $B\coloneqq \overline{g(E)}\subset V$. Let $r$ be the relative dimension of $g$. Since $g$ is flat, every fiber
has pure dimension $r$. Since
\[
\dim E=\dim Y-1=\dim V+r-1
\]
and the generic fiber of $E\to B$ has dimension at most $r$, we have
\[
\dim B\geq \dim V-1.
\]
Since $E$ is vertical, $B$ is a proper closed subset of $V$. Hence $B$ is
a prime divisor on $V$.

Since $V$ is smooth, the prime divisor $B$ is Cartier. By the choice of the open subset, the fiber over the generic point of $B$ is integral. Thus, $E$ is the unique prime divisor of $Y$ dominating $B$ and occurs with multiplicity one in $g^*B$. Therefore, $g^*B=E$ as Weil divisors, and hence $E$ is Cartier.

Let $D\subset Y$ be a horizontal prime divisor. For each $i$, choose a
Cartier divisor $H_i$ on $Y$ such that $\OO_Y(H_i)\simeq L_i$.
Since $Y_\eta$ is $\Q$-factorial and the classes of
$L_{1,\eta},\ldots,L_{\rho,\eta}$ form a basis of
$\Pic(Y_\eta)_{\Q}$, there exist a positive integer $m$, integers $a_i$,
and a rational function $\varphi\in K(Y_\eta)^\times=K(Y)^\times$ such that
\[
mD_\eta-\sum_{i=1}^{\rho}a_iH_{i,\eta}
=
\mathrm{div}_{Y_\eta}(\varphi).
\]
Since a horizontal prime divisor of $Y$ and its generic fiber define the same valuation on $K(Y)=K(Y_{\eta})$, it follows that
\[
mD-\sum_{i=1}^{\rho}a_iH_i-\mathrm{div}_Y(\varphi)
\]
is a vertical Weil divisor. Every vertical prime divisor on $Y$ is Cartier by the above argument. Thus, the vertical divisor above is Cartier. Hence, $mD$ is Cartier, and therefore $D$ is $\Q$-Cartier. This proves (2).

Since the classes of $L_{1,\eta},\ldots,L_{\rho,\eta}$ form a basis of
$\Pic(Y_\eta)_{\Q}$ and each $L_{i,\eta}$ extends to $L_i$ on $Y$, the
restriction map $\Pic(Y/V)_{\Q}\to\Pic(Y_\eta)_{\Q}$ is surjective. For injectivity, let $L$ be a line bundle on $Y$ whose restriction to $Y_\eta$ is $\Q$-linearly trivial. A rational trivialization of a positive power has vertical divisor. Every vertical divisor is pulled back from $V$. Consequently, the line bundle is trivial in $\Pic(Y/V)_{\Q}$.

Finally, let $M$ be a line bundle on $Y$ that is numerically trivial
over $V$. For an integral curve $C_\eta\subset Y_\eta$, let $\mathcal C\subset Y$
be its closure. By generic flatness, the projective morphism
$\mathcal C\to V$ is flat with one-dimensional fibers over a nonempty open subset of $V$. The difference of Euler characteristics of $M|_{\mathcal C}$ and $\OO_{\mathcal C}$ is constant in this flat family, and hence $\deg(M|_{\mathcal C_v})$ is independent of $v$. Relative numerical triviality gives $\deg(M|_{\mathcal C_v})=0$ for a general closed point $v\in V$, and therefore $\deg(M_\eta|_{C_\eta})=0$. Thus, $M_\eta$ is
numerically trivial.

Under the last assumption, $M_\eta$ is $\Q$-linearly trivial, and (3) implies that $M$ is relatively $\Q$-linearly trivial. Thus, the natural map $\Pic(Y/V)_{\Q}\to N^1(Y/V)_{\Q}$ is injective and hence an isomorphism.
\end{proof}

The next proposition passes from the geometric generic fiber to a relative Mori dream morphism.

\begin{proposition}\label{prop:generic-criterion}
Let $k$ be an uncountable algebraically closed field of arbitrary
characteristic, and let
$f\colon X\to T$ be a projective fibration between normal integral
$k$-varieties.  If the geometric generic fiber $X_{\overline\eta}$ is a
Mori dream space, then there is a nonempty open subset $U\subset T$ such
that $X_U\to U$ is a Mori dream morphism.
\end{proposition}

\begin{proof}
After replacing $T$ by a nonempty affine open subset, every variety
projective over $T$ is quasi-projective over $k$. Set $K\coloneqq k(T)$.
The algebraic closure $\overline K$ is uncountable since it contains $k$. By assumption, $X_{\overline\eta}$ satisfies Definition \ref{def:mdm}. By Lemma \ref{lem:picard-conditions}(2), both $X_{\overline\eta}$ and every small $\Q$-factorial modification in condition (4) of Definition \ref{def:mdm} have zero-dimensional Picard varieties. Hence, $X_{\overline\eta}$ is a Mori dream space in the sense of Okawa.
Proposition \ref{prop:okawa} then shows that the generic fiber $X_\eta$
is a Mori dream space over $K$ in the sense of Okawa. Choose finitely
many small $\Q$-factorial modifications
\[
 \phi_{i,\eta}\colon X_\eta\dashrightarrow X_{i,\eta},
 \qquad 1\leq i\leq r,
\]
such that
\[
 \Mov(X_\eta)=
 \bigcup_{i=1}^r\phi_{i,\eta}^*\Nef(X_{i,\eta}).
\]
We may assume that the identity map occurs among the $\phi_{i,\eta}$. Every variety
$X_{i,\eta}$ has function field $K(X_\eta)$. Since
$X_{\overline\eta}$ is integral, the extension $K(X_\eta)/K$ is regular.
Consequently, every $X_{i,\eta}$ is geometrically integral. For each $i$,
choose semiample line bundles $L_{ij,\eta}$ whose classes generate
$\Nef(X_{i,\eta})$, positive integers $m_{ij}$ such that
$L_{ij,\eta}^{\otimes m_{ij}}$ is globally generated, and finite sets of
global sections generating these powers.

Lemma \ref{lem:models-open} gives, after shrinking $T$ to a nonempty open subset $U$, projective morphisms $f_i\colon X_i\to U$ with $X_i$ normal, rational maps $\phi_i\colon X_U\dashrightarrow X_i$, extensions $L_{ij}$ of the line bundles $L_{ij,\eta}$, and extensions of the chosen generating sections. For the identity modification, we use
the restriction $X_U\to U$ of the given model $X\to T$. By
Lemma \ref{lem:small-open}, after a further shrinking, each $\phi_i$ is an
isomorphism in codimension one. Lemma \ref{lem:open-integral} allows us to
assume that $U$ is smooth and that every $X_i\to U$ is flat with
geometrically integral fibers. By \cite[Definition 2.1]{Okawa}, the varieties $X_\eta$ and
$X_{i,\eta}$ are $\Q$-factorial and satisfy
$\dim\Pic^0_{X_{i,\eta}/K}=0$ for every $i$, including the identity
modification. Lemma \ref{lem:picard-conditions}(1) therefore gives $\Pic(X_{i,\eta})_{\Q}=N^1(X_{i,\eta})_{\Q}$
for every $i$. By Lemma \ref{lem:qfactor-open}, after shrinking $U$ further, $X_U,X_1,\ldots,X_r$ are $\Q$-factorial, and for every $i$ the
restriction map induces an isomorphism $\Pic(X_i/U)_{\Q}\xrightarrow{\sim}\Pic(X_{i,\eta})_{\Q}$.
Since $\Pic(X_{i,\eta})_{\Q}=N^1(X_{i,\eta})_{\Q}$, we also have $\Pic(X_i/U)_{\Q}=N^1(X_i/U)_{\Q}$. Hence, the maps $\phi_i$ are small $\Q$-factorial modifications over $U$.

Since the generic fiber of $f_i$ is geometrically integral and $U$ is normal, Stein factorization gives $(f_i)_{\ast}\OO_{X_i}=\OO_U$.

The evaluation morphism associated with the extended sections of
$L_{ij}^{\otimes m_{ij}}$ is surjective on the
generic fiber. By the properness of $f_i$, after shrinking $U$, we may assume that every 
$L_{ij}^{\otimes m_{ij}}$ is relatively globally generated.
 
Lemma \ref{lem:qfactor-open} gives an isomorphism $N^1(X_i/U)_{\R}\xrightarrow{\sim}N^1(X_{i,\eta})_{\R}$
induced by restriction to the generic fiber. Under this isomorphism, every
relatively nef class on $X_i$ maps to a nef class on $X_{i,\eta}$. Conversely, let a class in $N^1(X_i/U)_{\R}$ restrict to a
nef class on $X_{i,\eta}$. The restricted class is a nonnegative linear
combination of the classes $[L_{ij,\eta}]$. Under the restriction
identification, the original relative class is the same combination of the
relatively semiample classes $[L_{ij}]$, and is therefore relatively nef.
Consequently, 
\[
 \Nef(X_i/U)=\Nef(X_{i,\eta}).
\]
Let $f_U\colon X_U\to U$ be the restriction of $f$, and let
$D$ be an $f_U$-movable Cartier divisor. Denote the evaluation morphism by
\[
 \varepsilon_D\colon
 f_U^*f_{U*}\OO_{X_U}(D)\longrightarrow\OO_{X_U}(D),
\]
and put $Z_D\coloneqq\Supp\mathrm{Coker}(\varepsilon_D)$.
The morphism $\Spec k(\eta)\to U$ is flat, and flat base change therefore gives
\[
 (f_{U*}\OO_{X_U}(D))_\eta
 \simeq H^0\bigl(X_\eta,\OO_{X_\eta}(D_\eta)\bigr),
\]
see \cite[Lemma 30.5.2, Tag 02KH]{Stacks}. Under this isomorphism, the restriction of $\varepsilon_D$ to $X_\eta$
is the evaluation morphism
\[
H^0\!\left(X_\eta,\OO_{X_\eta}(D_\eta)\right)
\otimes_K\OO_{X_\eta}
\longrightarrow
\OO_{X_\eta}(D_\eta).
\]
Hence, the support of its cokernel is $(Z_D)_\eta$.
Every irreducible component $W$ of $Z_D$ meeting $X_\eta$ dominates $U$.
Thus, $\codim_{X_\eta}W_\eta=\codim_{X_U}W\geq2$, and hence $D_\eta$ is movable.

We next prove that $\phi_i^*\Nef(X_i/U)\subseteq\Mov(X_U/U)$ for every $i$. Fix $i$ and $j$. Choose a Cartier divisor $H_{ij}$ on $X_i$ such that $\OO_{X_i}(H_{ij})\simeq L_{ij}$, and let $D_{ij}$ be the strict transform of $H_{ij}$ on $X_U$. Since $X_U$ is $\Q$-factorial, there is a positive integer $a_{ij}$, divisible by $m_{ij}$, such that $a_{ij}D_{ij}$ is Cartier. The line bundle $L_{ij}^{\otimes a_{ij}}$ is relatively globally generated on $X_i$. The map $\phi_i$ is an isomorphism at the generic point of every prime divisor. Therefore, the strict transform on $X_U$ of the complete linear system $|a_{ij}H_{ij}|$ has no divisorial fixed component. Hence, $a_{ij}D_{ij}$ is relatively movable, and
\[
 \phi_i^*[L_{ij}]
 =\frac{1}{a_{ij}}[a_{ij}D_{ij}]
 \in\Mov(X_U/U).
\]
Since the classes $[L_{ij}]$ generate $\Nef(X_i/U)$, we have $\phi_i^*\Nef(X_i/U)\subseteq\Mov(X_U/U)$. Under the isomorphisms induced by restriction to the generic fiber, pullback by $\phi_i$ corresponds to pullback by $\phi_{i,\eta}$. Therefore,
\[
\begin{aligned}
 \Mov(X_U/U)
 &\subseteq\Mov(X_\eta)\\
 &=\bigcup_{i=1}^r\phi_{i,\eta}^*\Nef(X_{i,\eta})\\
 &=\bigcup_{i=1}^r\phi_i^*\Nef(X_i/U)
 \subseteq\Mov(X_U/U).
\end{aligned}
\]
Hence, all the inclusions are equalities. Lemma \ref{lem:qfactor-open} and the equality of the nef cones give conditions (1)--(3) of Definition \ref{def:mdm} for $X_U\to U$, while the equality above gives (4). Hence, $X_U\to U$ is a Mori dream morphism.
\end{proof}

\subsection{Proof of Theorem \ref{thm:main}}\label{sec:proof-main}

\begin{proof}[Proof of Theorem \ref{thm:main}]
If $k$ is countable and $\dim T>0$, then $T(k)$ is countable, and every subset of $T(k)$ is contained in a countable union of closed points. Thus, the assumption on $S$ cannot hold. If $\dim T=0$, then $T=\Spec k$, and the conclusion follows directly from
the Mori dream property of the unique fiber. We may therefore assume that
$k$ is uncountable.

Lemma \ref{lem:countable-model} gives a countable algebraically closed
subfield $k_0\subset k$, a projective morphism $f_0\colon X_0\to T_0$ between normal integral $k_0$-varieties, and a point $s\in S$ whose image is the generic point of the geometrically integral variety $T_0$. Let $\eta_0$ denote the generic point of $T_0$. Both $X_s$ and the geometric generic fiber $X_{\overline\eta}$ are base
changes of the $k_0(T_0)$-variety $X_{0,\eta_0}$. By Lemma \ref{lem:field-comparison}, there is a
$k_0(T_0)$-isomorphism
\[
\sigma\colon k\xrightarrow{\sim}\overline{k(T)}
\]
compatible with the embeddings determined by $s$ and by the generic point.
Hence,
\[
X_s\times_{\Spec k,\sigma}\Spec\overline{k(T)}
\simeq X_{\overline\eta}.
\]
Since the Mori dream space property is invariant under an isomorphism of
the ground fields and under isomorphisms of varieties, $X_{\overline\eta}$ is also a Mori dream space.

Proposition \ref{prop:generic-criterion} then yields a nonempty open subset
$U\subset T$ such that $X_U\to U$ is a Mori dream morphism.

For the final assertion, every proper closed subset of the curve $T$ is finite. Hence, an uncountable set of points with Mori dream fibers cannot be contained in a countable union of proper closed subsets, and the first assertion completes the proof.
\end{proof}

\bibliographystyle{habbvr}
\bibliography{biblio}

\end{document}